\documentclass[12pt,fleqn]{article}

\usepackage{amsmath,amssymb,amsthm}
\usepackage{geometry} 
\usepackage[pdfpagemode=UseNone,pdfstartview=FitH,hypertex]{hyperref}

\newcommand{\bdry}[1]{\partial #1}
\newcommand{\bgset}[1]{\big\{#1\big\}}
\newcommand{\A}{{\cal A}}
\newcommand{\D}{{\cal D}}
\newcommand{\F}{{\cal F}}
\newcommand{\dint}{\ds{\int}}
\newcommand{\dist}[2]{\text{dist}\, (#1,#2)}
\newcommand{\ds}[1]{\displaystyle #1}
\newcommand{\eps}{\varepsilon}
\newcommand{\incl}{\hookrightarrow}
\newcommand{\M}{{\cal M}}
\newcommand{\N}{\mathbb N}
\newcommand{\norm}[2][]{\left\|#2\right\|_{#1}}
\renewcommand{\o}{\text{o}}
\newcommand{\PS}[1]{$(\text{PS})_{#1}$}
\newcommand{\pnorm}[2][]{\if #1'' \left|#2\right|_p \else \left|#2\right|_{#1} \fi}
\newcommand{\R}{\mathbb R}
\newcommand{\RP}{\R \text{P}}
\newcommand{\seq}[1]{\left(#1\right)}
\newcommand{\set}[1]{\left\{#1\right\}}
\newcommand{\vol}[1]{\left|#1\right|}
\newcommand{\Z}{\mathbb Z}

\DeclareMathOperator{\divg}{div}

\newenvironment{enumroman}{\begin{enumerate}

}{\end{enumerate}}

\newtheorem{corollary}{Corollary}[section]
\newtheorem{lemma}[corollary]{Lemma}
\newtheorem{proposition}[corollary]{Proposition}
\newtheorem{theorem}[corollary]{Theorem}

\theoremstyle{definition}
\newtheorem{definition}[corollary]{Definition}

\theoremstyle{remark}
\newtheorem{example}[corollary]{Example}
\newtheorem{remark}[corollary]{Remark}

\numberwithin{equation}{section}

\title{\bf Local and nonlocal critical growth anisotropic quasilinear elliptic systems}
\author{\bf Artur Jorge Marinho \& Kanishka Perera\\
Department of Mathematics\\
Florida Institute of Technology\\
150 W University Blvd, Melbourne, FL 32901-6975, USA\\
\em amarinho2024@my.fit.edu \& kperera@fit.edu}
\date{}

\begin{document}

\maketitle

\begin{abstract}
In this paper we prove new multiplicity results for a critical growth anisotropic quasilinear elliptic system that is coupled through a subcritical perturbation term. We identify a certain scaling for the system and a parameter $\gamma$ related to this scaling that determines the geometry of the associated variational functional. This leads to a natural classification of different nonlinear regimes for the system in terms of scaling properties of the perturbation term. We give three different types of multiplicity results in the three regimes $\gamma = 1$, $\gamma > 1$, and $\gamma < 1$. Proofs of our multiplicity results are based on a new abstract critical point theorem for symmetric functionals on product spaces, which we prove using the piercing property of the $\Z_2$-cohomological index of Fadell and Rabinowitz. This abstract result only requires a local \PS{} condition and is therefore applicable to systems with critical growth. It is of independent interest as it has wide applicability to many different types of critical elliptic systems. We also indicate how it can be applied to obtain similar multiplicity results for nonlocal systems.
\end{abstract}

\newpage

{\small \tableofcontents}

\section{Introduction}

In this paper we prove new multiplicity results for the critical growth anisotropic quasilinear elliptic system
\begin{equation} \label{101}
\left\{\begin{aligned}
- \Delta_p u & = \frac{\lambda a}{a + b}\, |u|^{a-2}\, |v|^b\, u + |u|^{p^\ast - 2}\, u && \text{in } \Omega\\[10pt]
- \Delta_q v & = \frac{\lambda b}{a + b}\, |u|^a\, |v|^{b-2}\, v + |v|^{q^\ast - 2}\, v && \text{in } \Omega\\[10pt]
u = v & = 0 && \text{on } \bdry{\Omega},
\end{aligned}\right.
\end{equation}
where $\Omega$ is a bounded domain in $\R^N$, $\Delta_p u = \divg\, (|\nabla u|^{p-2}\, \nabla u)$ is the $p$-Laplacian of $u$, $p, q \in (1,N)$, $p^\ast = Np/(N - p)$ and $q^\ast = Nq/(N - q)$ are the critical Sobolev exponents, $a, b > 1$ satisfy
\begin{equation} \label{102}
\frac{a}{p_0} + \frac{b}{q_0} = 1
\end{equation}
for some $p_0 \in (1,p^\ast)$ and $q_0 \in (1,q^\ast)$, and $\lambda > 0$ is a parameter. By \eqref{102} and the H\"{o}lder and Young's inequalities,
\begin{equation} \label{103}
\int_\Omega |u|^a\, |v|^b\, dx \le \left(\int_\Omega |u|^{p_0}\, dx\right)^{a/p_0}\! \left(\int_\Omega |v|^{q_0}\, dx\right)^{b/q_0} \le \frac{a}{p_0} \int_\Omega |u|^{p_0}\, dx + \frac{b}{q_0} \int_\Omega |v|^{q_0}\, dx,
\end{equation}
so system \eqref{101} may be thought of as a subcritical perturbation of the critical system corresponding to $\lambda = 0$.

Critical quasilinear systems in the isotropic case $p = q$ have been studied in, e.g., Carmona et al.\! \cite{MR3169759}, de Morais Filho and Souto \cite{MR1698537}, Ding and Xiao \cite{MR2577821}, and Hsu \cite{MR2532793}. On the other hand, subcritical quasilinear systems in the anisotropic case $p \ne q$ have been studied in, e.g., Bozhkov and Mitidieri \cite{MR1970963} and Candela et al.\! \cite{MR4250461}. However, critical anisotropic quasilinear systems do not seem to have been studied in the literature.

Proofs of our multiplicity results will be based on a new abstract critical point theorem for symmetric functionals on product spaces that we will prove in the next section (see Theorem \ref{Theorem 305}). This abstract result only requires a local \PS{} condition and is therefore applicable to systems with critical growth. It is of independent interest as it has wide applicability to many different types of critical elliptic systems. In particular, we will apply it to a related system in Section \ref{Subsection 1.4} and indicate how it can be applied to nonlocal systems in Section \ref{Subsection 1.5}.

We work in the product space
\[
W = W^{1,\,p}_0(\Omega) \times W^{1,\,q}_0(\Omega) = \set{w = (u,v) : u \in W^{1,\,p}_0(\Omega),\, v \in W^{1,\,q}_0(\Omega)}
\]
endowed with the norm
\[
\norm{w} = \norm[W^{1,\,p}_0(\Omega)]{u} + \norm[W^{1,\,q}_0(\Omega)]{v}.
\]
The variational functional associated with system \eqref{101} is given by
\begin{multline} \label{115}
E(w) = \int_\Omega \left(\frac{1}{p}\, |\nabla u|^p + \frac{1}{q}\, |\nabla v|^q\right) dx - \frac{\lambda}{a + b} \int_\Omega |u|^a\, |v|^b\, dx\\[7.5pt]
- \int_\Omega \left(\frac{1}{p^\ast}\, |u|^{p^\ast} + \frac{1}{q^\ast}\, |v|^{q^\ast}\right) dx, \quad w \in W.
\end{multline}
Consider the scaling
\[
W \times [0,\infty) \to W, \quad (w,t) \mapsto w_t = (t^{1/p}\, u,t^{1/q}\, v).
\]
We have
\begin{multline*}
E(w_t) = t \int_\Omega \left(\frac{1}{p}\, |\nabla u|^p + \frac{1}{q}\, |\nabla v|^q\right) dx - \frac{\lambda t^\gamma}{a + b} \int_\Omega |u|^a\, |v|^b\, dx\\[7.5pt]
- \int_\Omega \left(\frac{t^{p^\ast/p}}{p^\ast}\, |u|^{p^\ast} + \frac{t^{q^\ast/q}}{q^\ast}\, |v|^{q^\ast}\right) dx \quad \forall w \in W,\, t \ge 0,
\end{multline*}
where
\begin{equation} \label{104}
\gamma = \frac{a}{p} + \frac{b}{q}.
\end{equation}
The way that the term $\int_\Omega |u|^a\, |v|^b\, dx$ scales in comparison to the way that the leading term $\int_\Omega \left(\frac{1}{p}\, |\nabla u|^p + \frac{1}{q}\, |\nabla v|^q\right) dx$ scales will determine the type of multiplicity result that we prove, i.e., we will obtain different types of multiplicity results in the three cases $\gamma = 1$, $\gamma > 1$, and $\gamma < 1$.

The related eigenvalue problem
\begin{equation} \label{105}
\left\{\begin{aligned}
- \Delta_p u & = \frac{\lambda \alpha}{\alpha + \beta}\, |u|^{\alpha - 2}\, |v|^\beta\, u && \text{in } \Omega\\[10pt]
- \Delta_q v & = \frac{\lambda \beta}{\alpha + \beta}\, |u|^\alpha\, |v|^{\beta - 2}\, v && \text{in } \Omega\\[10pt]
u = v & = 0 && \text{on } \bdry{\Omega},
\end{aligned}\right.
\end{equation}
where $\alpha, \beta > 1$ satisfy
\begin{equation} \label{106}
\frac{\alpha}{p} + \frac{\beta}{q} = 1,
\end{equation}
will play a major role in our results and proofs. Note that if $w \in W \setminus \set{0}$ is an eigenfunction associated with an eigenvalue $\lambda$, then $w_t$ is also an eigenfunction associated with $\lambda$ for any $t > 0$ due to \eqref{106}. A sequence of minimax eigenvalues of this problem can be constructed in a standard way using the Krasnoselskii's genus. However, the proofs of our multiplicity results require a certain intersection property of the $\Z_2$-cohomological index (see Definition \ref{Definition 301} and property $(i_7)$ in Proposition \ref{Proposition 303}). The genus does not share this property, so we will work with a different sequence of eigenvalues $\seq{\lambda_k}$ constructed using the cohomological index (see Theorem \ref{Theorem 401}).

\subsection{The case $\gamma = 1$} \label{Subsection 1.1}

Let $\seq{\lambda_k}$ be the sequence of eigenvalues of the eigenvalue problem
\begin{equation} \label{107}
\left\{\begin{aligned}
- \Delta_p u & = \frac{\lambda a}{a + b}\, |u|^{a - 2}\, |v|^b\, u && \text{in } \Omega\\[10pt]
- \Delta_q v & = \frac{\lambda b}{a + b}\, |u|^a\, |v|^{b - 2}\, v && \text{in } \Omega\\[10pt]
u = v & = 0 && \text{on } \bdry{\Omega}
\end{aligned}\right.
\end{equation}
based on the cohomological index (see Theorem \ref{Theorem 401}). We will show that system \eqref{101} has $m$ distinct pairs of nontrivial solutions for all $\lambda$ in a suitably small left neighborhood of an eigenvalue of multiplicity $m \ge 1$.

We have
\[
\frac{a}{p^\ast} + \frac{b}{q^\ast} < \frac{a}{p_0} + \frac{b}{q_0} = 1
\]
by \eqref{102}. Let
\begin{equation} \label{113}
\frac{1}{\zeta} = 1 - \frac{a}{p^\ast} - \frac{b}{q^\ast}
\end{equation}
and let
\begin{equation} \label{515}
\delta = \frac{a + b}{a^{a/p^\ast} b^{b/q^\ast}} \left(\frac{\zeta}{N \vol{\Omega}}\, \min \set{S_p^{N/p},S_q^{N/q}}\right)^{1/\zeta},
\end{equation}
where $\vol{\Omega}$ denotes the volume of $\Omega$ and
\begin{equation} \label{201}
S_p = \inf_{u \in \D^{1,\,p}(\R^N) \setminus \set{0}}\, \frac{\int_{\R^N} |\nabla u|^p\, dx}{\left(\int_{\R^N} |u|^{p^\ast} dx\right)^{p/p^\ast}}, \qquad S_q = \inf_{u \in \D^{1,\,q}(\R^N) \setminus \set{0}}\, \frac{\int_{\R^N} |\nabla u|^q\, dx}{\left(\int_{\R^N} |u|^{q^\ast} dx\right)^{q/q^\ast}}.
\end{equation}

\begin{theorem} \label{Theorem 101}
Assume that $\gamma = 1$. If $\lambda_k = \cdots = \lambda_{k+m-1} < \lambda_{k+m}$ for some $k, m \ge 1$, then for all $\lambda \in (\lambda_k - \delta,\lambda_k)$, system \eqref{101} has $m$ distinct pairs of nontrivial solutions $\pm w^\lambda_j,\, j = 1,\dots,m$ at positive energy levels such that each $w^\lambda_j \to 0$ as $\lambda \nearrow \lambda_k$.
\end{theorem}

In particular, we have the following corollary when $m = 1$ (since $\lambda_k \nearrow \infty$, by taking $k$ larger if necessary, we may assume that $\lambda_k < \lambda_{k+1}$).

\begin{corollary}
If $\gamma = 1$, then for each $k \ge 1$, system \eqref{101} has a nontrivial solution for all $\lambda \in (\lambda_k - \delta,\lambda_k)$.
\end{corollary}

\begin{remark}
A similar result for a single equation was proved in Cerami et al.\! \cite{MR779872} in the semilinear case $p = 2$ and in Perera et al.\! \cite{MR3469053} in the general case $1 < p < N$.
\end{remark}

\subsection{The case $\gamma > 1$}

We will show that system \eqref{101} has arbitrarily many solutions for all sufficiently large $\lambda$ when $\gamma > 1$.

\begin{theorem} \label{Theorem 104}
Assume that \eqref{102} holds and $\gamma > 1$. Then for any $m \ge 1$, $\exists \lambda_m^\ast > 0$ such that for all $\lambda > \lambda_m^\ast$, system \eqref{101} has $m$ distinct pairs of nontrivial solutions $\pm w^\lambda_j,\, j = 1,\dots,m$ at positive energy levels such that each $w^\lambda_j \to 0$ as $\lambda \to \infty$. In particular, the number of solutions goes to infinity as $\lambda \to \infty$.
\end{theorem}

If $p_0 \in (p,p^\ast)$ and $q_0 \in (q,q^\ast)$ in \eqref{102}, then
\[
\gamma = \frac{a}{p} + \frac{b}{q} > \frac{a}{p_0} + \frac{b}{q_0} = 1,
\]
so we have the following corollary.

\begin{corollary}
If \eqref{102} holds for some $p_0 \in (p,p^\ast)$ and $q_0 \in (q,q^\ast)$, then for any $m \ge 1$, $\exists \lambda_m^\ast > 0$ such that for all $\lambda > \lambda_m^\ast$, system \eqref{101} has $m$ distinct pairs of nontrivial solutions $\pm w^\lambda_j,\, j = 1,\dots,m$ at positive energy levels such that each $w^\lambda_j \to 0$ as $\lambda \to \infty$. In particular, the number of solutions goes to infinity as $\lambda \to \infty$.
\end{corollary}

\begin{remark}
A similar result for a single equation was proved in Perera \cite{Pe23}.
\end{remark}

\subsection{The case $\gamma < 1$}

We will show that system \eqref{101} has infinitely many solutions for all sufficiently small $\lambda$ when $\gamma < 1$.

\begin{theorem} \label{Theorem 106}
Assume that \eqref{102} holds and $\gamma < 1$. Then $\exists \lambda_\ast > 0$ such that for all $\lambda \in (0,\lambda_\ast)$, system \eqref{101} has a sequence of solutions $\seq{w^\lambda_k}$ at negative energy levels such that each $w^\lambda_k \to 0$ as $\lambda \searrow 0$.
\end{theorem}

If $p_0 \in (1,p)$ and $q_0 \in (1,q)$ in \eqref{102}, then
\[
\gamma = \frac{a}{p} + \frac{b}{q} < \frac{a}{p_0} + \frac{b}{q_0} = 1,
\]
so we have the following corollary.

\begin{corollary}
If \eqref{102} holds for some $p_0 \in (1,p)$ and $q_0 \in (1,q)$, then $\exists \lambda_\ast > 0$ such that for all $\lambda \in (0,\lambda_\ast)$, system \eqref{101} has a sequence of solutions $\seq{w^\lambda_k}$ at negative energy levels such that each $w^\lambda_k \to 0$ as $\lambda \searrow 0$.
\end{corollary}

\begin{remark}
A similar result for a single equation was proved in Garc{\'{\i}}a Azorero and Peral Alonso \cite{MR1083144}.
\end{remark}

\subsection{A related system} \label{Subsection 1.4}

Next we consider the related system
\begin{equation} \label{109}
\left\{\begin{aligned}
- \Delta_p u & = \frac{\lambda a}{a + b}\, |u|^{a-2}\, |v|^b\, u - \frac{\mu c}{c + d}\, |u|^{c-2}\, |v|^d\, u + |u|^{p^\ast - 2}\, u && \text{in } \Omega\\[10pt]
- \Delta_q v & = \frac{\lambda b}{a + b}\, |u|^a\, |v|^{b-2}\, v - \frac{\mu d}{c + d}\, |u|^c\, |v|^{d-2}\, v + |v|^{q^\ast - 2}\, v && \text{in } \Omega\\[10pt]
u = v & = 0 && \text{on } \bdry{\Omega},
\end{aligned}\right.
\end{equation}
where $\Omega$ is a bounded domain in $\R^N$, $p, q \in (1,N)$, $a, b, c, d > 1$ satisfy
\begin{equation} \label{110}
\frac{a}{p} + \frac{b}{q} = 1
\end{equation}
and
\begin{equation} \label{111}
\frac{c}{p} + \frac{d}{q} = \kappa < 1,
\end{equation}
and $\lambda, \mu > 0$ are parameters. When $\mu = 0$ this system reduces to the system \eqref{101} with $\gamma = 1$ that was considered in Section \ref{Subsection 1.1}. However, a new phenomenon appears when $\mu > 0$, namely, if either $\mu$ or the volume of $\Omega$ is sufficiently small, then the number of solutions goes to infinity as $\lambda \to \infty$.

As in Section \ref{Subsection 1.1}, let $\seq{\lambda_k}$ be the sequence of eigenvalues of the eigenvalue problem \eqref{107} based on the cohomological index. We have
\[
\frac{c}{p^\ast} + \frac{d}{q^\ast} < \frac{c}{p} + \frac{d}{q} < 1
\]
by \eqref{111}. Let
\begin{equation} \label{114}
\frac{1}{\eta} = 1 - \frac{c}{p^\ast} - \frac{d}{q^\ast}.
\end{equation}
\begin{theorem} \label{Theorem 110}
Assume that \eqref{110} and \eqref{111} hold. If
\begin{equation} \label{112}
\mu \vol{\Omega}^{1/\eta} < \frac{c + d}{c^{c/p^\ast} d^{d/q^\ast}} \left(\frac{\eta}{N}\, \min \set{S_p^{N/p},S_q^{N/q}}\right)^{1/\eta},
\end{equation}
then for all $\lambda > \lambda_k$, system \eqref{109} has $k$ distinct pairs of nontrivial solutions $\pm w^{\lambda,\,\mu}_j,\, j = 1,\dots,k$ at positive energy levels such that each $w^{\lambda,\,\mu}_j \to 0$ as $\mu \searrow 0$. In particular, the number of solutions goes to infinity as $\lambda \to \infty$.
\end{theorem}

\begin{remark}
A similar result for a single equation was proved in El Manouni and Perera \cite{MaPe5}.
\end{remark}

\subsection{Nonlocal systems} \label{Subsection 1.5}

Finally we briefly indicate how our results can be extended to the nonlocal critical growth anisotropic quasilinear elliptic system
\begin{equation} \label{117}
\left\{\begin{aligned}
(- \Delta)_p^s\, u & = \frac{\lambda a}{a + b}\, |u|^{a-2}\, |v|^b\, u + |u|^{p_s^\ast - 2}\, u && \text{in } \Omega\\[10pt]
(- \Delta)_q^r\, v & = \frac{\lambda b}{a + b}\, |u|^a\, |v|^{b-2}\, v + |v|^{q_r^\ast - 2}\, v && \text{in } \Omega\\[10pt]
u = v & = 0 && \text{in } \R^N \setminus \Omega,
\end{aligned}\right.
\end{equation}
where $\Omega$ is a bounded domain in $\R^N$ with Lipschitz boundary, $(- \Delta)_p^s$ is the fractional $p$\nobreakdash-Laplacian operator defined on smooth functions by
\[
(- \Delta)_p^s\, u(x) = 2 \lim_{\eps \searrow 0} \int_{\R^N \setminus B_\eps(x)} \frac{|u(x) - u(y)|^{p-2}\, (u(x) - u(y))}{|x - y|^{N+sp}}\, dy, \quad x \in \R^N,
\]
$r, s \in (0,1)$, $1 < p < N/s$ and $1 < q < N/r$, $p_s^\ast = Np/(N - sp)$ and $q_r^\ast = Nq/(N - rq)$ are the fractional critical Sobolev exponents, $a, b > 1$ satisfy
\begin{equation} \label{120}
\frac{a}{p_0} + \frac{b}{q_0} = 1
\end{equation}
for some $p_0 \in (1,p_s^\ast)$ and $q_0 \in (1,q_r^\ast)$, and $\lambda > 0$ is a parameter.

Let
\[
[u]_{s,\,p} = \left(\int_{\R^{2N}} \frac{|u(x) - u(y)|^p}{|x - y|^{N+sp}}\, dx dy\right)^{1/p}
\]
be the Gagliardo seminorm of a measurable function $u : \R^N \to \R$ and let
\[
W^{s,\,p}(\R^N) = \set{u \in L^p(\R^N) : [u]_{s,\,p} < \infty}
\]
be the fractional Sobolev space endowed with the norm
\[
\norm[s,\,p]{u} = \left(\pnorm{u}^p + [u]_{s,\,p}^p\right)^{1/p},
\]
where $\pnorm{\cdot}$ denotes the norm in $L^p(\R^N)$. Recall that
\[
W^{s,\,p}_0(\Omega) = \set{u \in W^{s,\,p}(\R^N) : u = 0 \text{ a.e.\! in } \R^N \setminus \Omega},
\]
equivalently renormed by setting $\norm[W^{s,\,p}_0(\Omega)]{\cdot} = [\cdot]_{s,\,p}$.

We work in the product space
\[
W = W^{s,\,p}_0(\Omega) \times W^{r,\,q}_0(\Omega) = \set{w = (u,v) : u \in W^{s,\,p}_0(\Omega),\, v \in W^{r,\,q}_0(\Omega)}
\]
endowed with the norm
\[
\norm{w} = \norm[W^{s,\,p}_0(\Omega)]{u} + \norm[W^{r,\,q}_0(\Omega)]{v}.
\]
The variational functional associated with system \eqref{117} is
\begin{multline} \label{121}
E(w) = \int_{\R^{2N}} \left(\frac{1}{p}\, \frac{|u(x) - u(y)|^p}{|x - y|^{N+sp}} + \frac{1}{q}\, \frac{|v(x) - v(y)|^q}{|x - y|^{N+rq}}\right) dx dy - \frac{\lambda}{a + b} \int_\Omega |u|^a\, |v|^b\, dx\\[7.5pt]
- \int_\Omega \left(\frac{1}{p_s^\ast}\, |u|^{p_s^\ast} + \frac{1}{q_r^\ast}\, |v|^{q_r^\ast}\right) dx, \quad w \in W.
\end{multline}
With respect to the scaling $W \times [0,\infty) \to W,\, (w,t) \mapsto w_t = (t^{1/p}\, u,t^{1/q}\, v)$, we have
\begin{multline*}
E(w_t) = t \int_{\R^{2N}} \left(\frac{1}{p}\, \frac{|u(x) - u(y)|^p}{|x - y|^{N+sp}} + \frac{1}{q}\, \frac{|v(x) - v(y)|^q}{|x - y|^{N+rq}}\right) dx dy - \frac{\lambda t^\gamma}{a + b} \int_\Omega |u|^a\, |v|^b\, dx\\[7.5pt]
- \int_\Omega \left(\frac{t^{p_s^\ast/p}}{p_s^\ast}\, |u|^{p_s^\ast} + \frac{t^{q_r^\ast/q}}{q_r^\ast}\, |v|^{q_r^\ast}\right) dx \quad \forall w \in W,\, t \ge 0,
\end{multline*}
where $\gamma = a/p + b/q$. As in the local case, we obtain different types of multiplicity results in the three cases $\gamma = 1$, $\gamma > 1$, and $\gamma < 1$.

Let
\[
S_{p,s} = \inf_{u \in \dot{W}^{s,\,p}(\R^N) \setminus \set{0}}\, \frac{\dint_{\R^{2N}} \frac{|u(x) - u(y)|^p}{|x - y|^{N+sp}}\, dx dy}{\left(\dint_{\R^N} |u|^{p_s^\ast}\, dx\right)^{p/p_s^\ast}}
\]
be the best fractional Sobolev constant, where
\[
\dot{W}^{s,\,p}(\R^N) = \set{u \in L^{p_s^\ast}(\R^N) : [u]_{s,\,p} < \infty}
\]
endowed with the norm $\norm[\dot{W}^{s,\,p}(\R^N)]{\cdot} = [\cdot]_{s,\,p}$, and let
\[
c^\ast = \frac{1}{N}\, \min \set{s\, S_{p,s}^{N/sp},r\, S_{q,r}^{N/rq}}.
\]
Then an argument similar to that in the proof of Lemma \ref{Lemma 201} shows that the functional $E$ in \eqref{121} satisfies the \PS{c} condition for all $c < c^\ast$ if $\gamma \ge 1$.

When $\gamma = 1$, the corresponding eigenvalue problem is
\begin{equation} \label{119}
\left\{\begin{aligned}
(- \Delta)_p^s\, u & = \frac{\lambda a}{a + b}\, |u|^{a - 2}\, |v|^b\, u && \text{in } \Omega\\[10pt]
(- \Delta)_q^r\, v & = \frac{\lambda b}{a + b}\, |u|^a\, |v|^{b - 2}\, v && \text{in } \Omega\\[10pt]
u = v & = 0 && \text{in } \R^N \setminus \Omega.
\end{aligned}\right.
\end{equation}
Let
\begin{multline*}
I(w) = \int_{\R^{2N}} \left(\frac{1}{p}\, \frac{|u(x) - u(y)|^p}{|x - y|^{N+sp}} + \frac{1}{q}\, \frac{|v(x) - v(y)|^q}{|x - y|^{N+rq}}\right) dx dy,\\[7.5pt]
J(w) = \frac{1}{a + b} \int_\Omega |u|^a\, |v|^b\, dx, \quad w \in W,
\end{multline*}
let $\M = \set{w \in W : I(w) = 1}$, and let $\M^+ = \set{w \in \M : J(w) > 0}$. Then eigenvalues of problem \eqref{119} coincide with critical values of the functional
\[
\Psi(w) = \frac{1}{J(w)}, \quad w \in \M^+,
\]
and Theorem \ref{Theorem 401} holds for this eigenvalue problem also. Let
\[
\delta = \frac{a + b}{a^{a/p_s^\ast}\, b^{b/q_r^\ast}} \left(\frac{\zeta}{N \vol{\Omega}}\, \min \set{s\, S_{p,s}^{N/sp},r\, S_{q,r}^{N/rq}}\right)^{1/\zeta},
\]
where $1/\zeta = 1 - a/p_s^\ast - b/q_r^\ast$. Arguing as in the local case, we can prove the following theorems and corollaries.

\begin{theorem}
Assume that $\gamma = 1$. If $\lambda_k = \cdots = \lambda_{k+m-1} < \lambda_{k+m}$ for some $k, m \ge 1$, then for all $\lambda \in (\lambda_k - \delta,\lambda_k)$, system \eqref{117} has $m$ distinct pairs of nontrivial solutions $\pm w^\lambda_j,\, j = 1,\dots,m$ at positive energy levels such that each $w^\lambda_j \to 0$ as $\lambda \nearrow \lambda_k$.
\end{theorem}

\begin{corollary}
If $\gamma = 1$, then for each $k \ge 1$, system \eqref{117} has a nontrivial solution for all $\lambda \in (\lambda_k - \delta,\lambda_k)$.
\end{corollary}

\begin{theorem}
Assume that \eqref{120} holds and $\gamma > 1$. Then for any $m \ge 1$, $\exists \lambda_m^\ast > 0$ such that for all $\lambda > \lambda_m^\ast$, system \eqref{117} has $m$ distinct pairs of nontrivial solutions $\pm w^\lambda_j,\, j = 1,\dots,m$ at positive energy levels such that each $w^\lambda_j \to 0$ as $\lambda \to \infty$. In particular, the number of solutions goes to infinity as $\lambda \to \infty$.
\end{theorem}

\begin{corollary}
If \eqref{120} holds for some $p_0 \in (p,p^\ast)$ and $q_0 \in (q,q^\ast)$, then for any $m \ge 1$, $\exists \lambda_m^\ast > 0$ such that for all $\lambda > \lambda_m^\ast$, system \eqref{117} has $m$ distinct pairs of nontrivial solutions $\pm w^\lambda_j,\, j = 1,\dots,m$ at positive energy levels such that each $w^\lambda_j \to 0$ as $\lambda \to \infty$. In particular, the number of solutions goes to infinity as $\lambda \to \infty$.
\end{corollary}

\begin{theorem}
Assume that \eqref{120} holds and $\gamma < 1$. Then $\exists \lambda_\ast > 0$ such that for all $\lambda \in (0,\lambda_\ast)$, system \eqref{117} has a sequence of solutions $\seq{w^\lambda_k}$ at negative energy levels such that each $w^\lambda_k \to 0$ as $\lambda \searrow 0$.
\end{theorem}

\begin{corollary}
If \eqref{120} holds for some $p_0 \in (1,p)$ and $q_0 \in (1,q)$, then $\exists \lambda_\ast > 0$ such that for all $\lambda \in (0,\lambda_\ast)$, system \eqref{117} has a sequence of solutions $\seq{w^\lambda_k}$ at negative energy levels such that each $w^\lambda_k \to 0$ as $\lambda \searrow 0$.
\end{corollary}

\section{An abstract multiplicity result in product spaces} \label{Section 3}

In this section we prove the abstract critical point theorem in product spaces that we will use to prove our multiplicity results. This theorem only requires a local \PS{} condition and is therefore applicable to problems with critical growth. It will be based on the $\Z_2$-cohomological index of Fadell and Rabinowitz \cite{MR0478189}.

\begin{definition}[see \cite{MR0478189}] \label{Definition 301}
Let $\A$ denote the class of symmetric subsets of $W \setminus \set{0}$. For $A \in \A$, let $\overline{A} = A/\Z_2$ be the quotient space of $A$ with each $u$ and $-u$ identified, let $f : \overline{A} \to \RP^\infty$ be the classifying map of $\overline{A}$, and let $f^\ast : H^\ast(\RP^\infty) \to H^\ast(\overline{A})$ be the induced homomorphism of the Alexander-Spanier cohomology rings. The cohomological index of $A$ is defined by
\[
i(A) = \begin{cases}
0 & \text{if } A = \emptyset\\[5pt]
\sup \set{m \ge 1 : f^\ast(\omega^{m-1}) \ne 0} & \text{if } A \ne \emptyset,
\end{cases}
\]
where $\omega \in H^1(\RP^\infty)$ is the generator of the polynomial ring $H^\ast(\RP^\infty) = \Z_2[\omega]$.
\end{definition}

\begin{example}
The classifying map of the unit sphere $S^N$ in $\R^{N+1},\, N \ge 0$ is the inclusion $\RP^N \incl \RP^\infty$, which induces isomorphisms on the cohomology groups $H^l$ for $l \le N$, so $i(S^N) = N + 1$.
\end{example}

The following proposition summarizes the basic properties of the cohomological index.

\begin{proposition}[see \cite{MR0478189}] \label{Proposition 303}
The index $i : \A \to \N \cup \set{0,\infty}$ has the following properties:
\begin{enumerate}
\item[$(i_1)$] Definiteness: $i(A) = 0$ if and only if $A = \emptyset$.
\item[$(i_2)$] Monotonicity: If there is an odd continuous map from $A$ to $B$ (in particular, if $A \subset B$), then $i(A) \le i(B)$. Thus, equality holds when the map is an odd homeomorphism.
\item[$(i_3)$] Dimension: $i(A) \le \dim W$.
\item[$(i_4)$] Continuity: If $A$ is closed, then there is a closed neighborhood $N \in \A$ of $A$ such that $i(N) = i(A)$. When $A$ is compact, $N$ may be chosen to be a $\delta$-neighborhood $N_\delta(A) = \set{u \in W : \dist{u}{A} \le \delta}$.
\item[$(i_5)$] Subadditivity: If $A$ and $B$ are closed, then $i(A \cup B) \le i(A) + i(B)$.
\item[$(i_6)$] Stability: If $\Sigma A$ is the suspension of $A \ne \emptyset$, obtained as the quotient space of $A \times [-1,1]$ with $A \times \set{1}$ and $A \times \set{-1}$ collapsed to different points, then $i(\Sigma A) = i(A) + 1$.
\item[$(i_7)$] Piercing property: If $C$, $C_0$, and $C_1$ are closed and $\varphi : C \times [0,1] \to C_0 \cup C_1$ is a continuous map such that $\varphi(-u,t) = - \varphi(u,t)$ for all $(u,t) \in C \times [0,1]$, $\varphi(C \times [0,1])$ is closed, $\varphi(C \times \set{0}) \subset C_0$, and $\varphi(C \times \set{1}) \subset C_1$, then $i(\varphi(C \times [0,1]) \cap C_0 \cap C_1) \ge i(C)$.
\item[$(i_8)$] Neighborhood of zero: If $U$ is a bounded closed symmetric neighborhood of $0$, then $i(\bdry{U}) = \dim W$.
\end{enumerate}
\end{proposition}

To state our abstract multiplicity result, let $(W_1,\norm[1]{\cdot})$ and $(W_2,\norm[2]{\cdot})$ be Banach spaces and let
\[
W = W_1 \times W_2 = \set{w = (u,v) : u \in W_1,\, v \in W_2}
\]
be their product endowed with the norm
\[
\norm{w} = \norm[1]{u} + \norm[2]{v}.
\]
Consider the scaling
\[
W \times [0,\infty) \to W, \quad (w,t) \mapsto w_t = (t^\mu u,t^\nu v),
\]
where $\mu, \nu > 0$. Let $I \in C(W,\R)$ be an even functional satisfying
\begin{equation} \label{301}
I(w_t) = tI(w) \quad \forall w \in W,\, t \ge 0
\end{equation}
and
\begin{equation} \label{302}
I(w) > 0 \quad \forall w \in W \setminus \set{0},
\end{equation}
and let $\M = \set{w \in W : I(w) = 1}$. The set $\M$ is closed and symmetric since $I$ is continuous and even. We assume that $\M$ is bounded. In view of \eqref{301} and \eqref{302}, we can define a continuous projection $\pi : W \setminus \set{0} \to \M$ by
\begin{equation} \label{303}
\pi(w) = w_{t_w},
\end{equation}
where $t_w = I(w)^{-1}$.

Let $E \in C^1(W,\R)$ be an even functional. Assume that $\exists c^\ast > 0$ such that $E$ satisfies the \PS{c} condition for all $c \in (0,c^\ast)$. Let $\Gamma$ denote the group of odd homeomorphisms of $W$ that are the identity outside $E^{-1}(0,c^\ast)$, let $\A^\ast$ denote the class of symmetric subsets of $W$, and let
\[
\M_\rho = \set{w \in W : I(w) = \rho} = \set{w_\rho : w \in \M}
\]
for $\rho > 0$.

\begin{definition}[see \cite{MR84c:58014}]
The pseudo-index of $M \in \A^\ast$ related to $i$, $\M_\rho$, and $\Gamma$ is defined \linebreak by
\begin{equation} \label{304}
i^\ast(M) = \min_{\gamma \in \Gamma}\, i(\gamma(M) \cap \M_\rho).
\end{equation}
\end{definition}

We have the following theorem.

\begin{theorem} \label{Theorem 305}
Let $A_0$ and $B_0$ be symmetric subsets of $\M$ such that $A_0$ is compact, $B_0$ is closed, and
\begin{equation} \label{305}
i(A_0) \ge k + m - 1, \qquad i(\M \setminus B_0) \le k - 1
\end{equation}
for some $k, m \ge 1$. Let $R > \rho > 0$ and let
\begin{gather*}
X = \set{w_t : w \in A_0,\, 0 \le t \le R},\\
A = \set{w_R : w \in A_0},\\
B = \set{w_\rho : w \in B_0}.
\end{gather*}
Assume that
\begin{equation} \label{306}
\sup_{w \in A}\, E(w) \le 0 < \inf_{w \in B}\, E(w), \qquad \sup_{w \in X}\, E(w) < c^\ast.
\end{equation}
For $j = k,\dots,k + m - 1$, let $\A_j^\ast = \set{M \in \A^\ast : M \text{ is compact and } i^\ast(M) \ge j}$ and set
\[
c_j^\ast := \inf_{M \in \A_j^\ast}\, \sup_{w \in M}\, E(w).
\]
Then
\[
\inf_{w \in B}\, E(w) \le c_k^\ast \le \dotsb \le c_{k+m-1}^\ast \le \sup_{w \in X}\, E(w),
\]
each $c_j^\ast$ is a critical value of $E$, and $E$ has $m$ distinct pairs of associated critical points.
\end{theorem}

\begin{proof}
First we show that $c_k^\ast > 0$. Let $M \in \A_k^\ast$. Since the identity map is in $\Gamma$, \eqref{304} gives
\begin{equation} \label{307}
i(M \cap \M_\rho) \ge i^\ast(M) \ge k.
\end{equation}
On the other hand, since the restriction of $\pi$ to $\M_\rho \setminus B$ is an odd homeomorphism onto $\M \setminus B_0$, Proposition \ref{Proposition 303} $(i_2)$ gives
\begin{equation} \label{308}
i(\M_\rho \setminus B) = i(\M \setminus B_0) \le k - 1
\end{equation}
by \eqref{305}. Combining \eqref{307} and \eqref{308} gives $i(M \cap \M_\rho) > i(\M_\rho \setminus B)$, so $M$ intersects $B$ by Proposition \ref{Proposition 303} $(i_2)$ again. It follows that
\[
c_k^\ast \ge \inf_{w \in B}\, E(w) > 0
\]
by \eqref{306}.

Next we show that $c_{k+m-1}^\ast < c^\ast$. Let $\gamma \in \Gamma$. Since $A_0$ is compact, so are $X$ and $A$, in particular, $A$ is closed. Consider the continuous map $\varphi : A \times [0,1] \to W$ defined by
\[
\varphi(w,t) = \gamma(w_t).
\]
Since $\gamma$ is odd,
\[
\varphi(-w,t) = \gamma((-w)_t) = \gamma(-w_t) = - \gamma(w_t) = - \varphi(w,t) \quad \forall (w,t) \in A \times [0,1].
\]
We have $\varphi(A \times [0,1]) = \gamma(X)$. Since $X$ is compact, so is $\gamma(X)$, so $\varphi(A \times [0,1])$ is closed. Since $w_0 = 0$ for all $w \in W$ and $\gamma$ is odd,
\[
\varphi(A \times \set{0}) = \set{\gamma(0)} = \set{0}.
\]
On the other hand, $w_1 = w$ for all $w \in W$, and $\gamma$ is the identity on $A$ since $E \le 0$ on $A$ by \eqref{306} and $\gamma$ is the identity outside $E^{-1}(0,c^\ast)$, so
\[
\varphi(A \times \set{1}) = \gamma(A) = A.
\]
Noting that
\[
I(w) = R > \rho \quad \forall w \in A
\]
by \eqref{301} and applying Proposition \ref{Proposition 303} $(i_7)$ with $C = A$, $C_0 = \set{w \in W : I(w) \le \rho}$, and $C_1 = \set{w \in W : I(w) \ge \rho}$ now gives
\begin{equation} \label{309}
i(\gamma(X) \cap \M_\rho) = i(\varphi(A \times [0,1]) \cap C_0 \cap C_1) \ge i(A).
\end{equation}
Since the restriction of $\pi$ to $A$ is an odd homeomorphism onto $A_0$, Proposition \ref{Proposition 303} $(i_2)$ gives
\begin{equation} \label{310}
i(A) = i(A_0) \ge k + m - 1
\end{equation}
by \eqref{305}. Combining \eqref{309} and \eqref{310} gives $i(\gamma(X) \cap \M_\rho) \ge k + m - 1$. Since $\gamma \in \Gamma$ is arbitrary, it follows that $i^\ast(X) \ge k + m - 1$. So $X \in \A_{k+m-1}^\ast$ and hence
\[
c_{k+m-1}^\ast \le \sup_{w \in X}\, E(w) < c^\ast
\]
by \eqref{306}.

The rest now follows from standard results in critical point theory (see, e.g., Perera et al.\! \cite[Proposition 3.42]{MR2640827}).
\end{proof}

In particular, we have the following corollary when $B_0 = \M$ and $k = 1$ (see also Perera \cite{Pe23}).

\begin{corollary} \label{Corollary 306}
Let $A_0$ be a compact symmetric subset of $\M$ with $i(A_0) = m \ge 1$, let $R > \rho > 0$, and let
\[
A = \set{w_R : w \in A_0}, \qquad X = \set{w_t : w \in A_0,\, 0 \le t \le R}.
\]
Assume that
\begin{equation} \label{311}
\sup_{w \in A}\, E(w) \le 0 < \inf_{w \in \M_\rho}\, E(w), \qquad \sup_{w \in X}\, E(w) < c^\ast.
\end{equation}
Then $E$ has $m$ distinct pairs of critical points $\pm w_j,\, j = 1,\dots,m$ such that
\[
\inf_{w \in \M_\rho}\, E(w) \le E(w_j) \le \sup_{w \in X}\, E(w)
\]
for each $j$.
\end{corollary}

\section{Eigenvalue problem}

In this section we recall the construction of variational eigenvalues based on the cohomological index for the eigenvalue problem \eqref{105}. Let
\[
I(w) = \int_\Omega \left(\frac{1}{p}\, |\nabla u|^p + \frac{1}{q}\, |\nabla v|^q\right) dx, \quad w \in W
\]
and let $\M = \set{w \in W : I(w) = 1}$. Then $\M$ is a bounded complete symmetric $C^1$-Finsler manifold. Let
\[
J(w) = \frac{1}{\alpha + \beta} \int_\Omega |u|^\alpha\, |v|^\beta\, dx, \quad w \in W
\]
and let $\M^+ = \set{w \in \M : J(w) > 0}$. Then $\M^+$ is a symmetric open submanifold of $\M$ and eigenvalues of problem \eqref{105} coincide with critical values of the $C^1$-functional
\[
\Psi(w) = \frac{1}{J(w)}
\]
on $\M^+$ (see Perera et al.\! \cite[Lemma 10.7]{MR2640827}). Let $\F$ denote the class of symmetric subsets of $\M^+$ and let $i(M)$ be the cohomological index of $M \in \F$ (see Definition \ref{Definition 301}). The following theorem was proved in \cite{MR2640827}.

\begin{theorem}[see {\cite[Theorem 10.10]{MR2640827}}] \label{Theorem 401}
For $k \ge 1$, let $\F_k = \set{M \in \F : i(M) \ge k}$ and set
\[
\lambda_k := \inf_{M \in \F_k}\, \sup_{w \in M}\, \Psi(w).
\]
Then $\lambda_k \nearrow \infty$ is a sequence of eigenvalues of problem \eqref{105}.
\begin{enumroman}
\item \label{Theorem 401.i} The first eigenvalue is given by
    \[
    \lambda_1 = \min_{w \in \M^+}\, \Psi(w) > 0.
    \]
\item If $\lambda_k = \dotsb = \lambda_{k+m-1} = \lambda$ and $E_\lambda$ is the set of eigenfunctions associated with $\lambda$ that lie on $\M^+$, then
    \[
    i(E_\lambda) \ge m.
    \]
\item \label{Theorem 401.iii} If $\lambda_k < \lambda < \lambda_{k+1}$, then
    \[
    i(\Psi^{\lambda_k}) = i(\M^+ \setminus \Psi_\lambda) = i(\Psi^\lambda) = i(\M^+ \setminus \Psi_{\lambda_{k+1}}) = k,
    \]
    where $\Psi^a = \set{w \in \M^+ : \Psi(w) \le a}$ and $\Psi_a = \set{w \in \M^+ : \Psi(w) \ge a}$ for $a \in \R$.
\end{enumroman}
\end{theorem}

\section{Local Palais-Smale condition}

In this section we prove a local Palais-Smale condition for the functional $E$ defined in \eqref{115} when $\gamma \ge 1$. Let
\begin{equation} \label{202}
c^\ast = \frac{1}{N}\, \min \set{S_p^{N/p},S_q^{N/q}}.
\end{equation}

\begin{lemma} \label{Lemma 201}
If $\gamma \ge 1$, then $E$ satisfies the {\em \PS{c}} condition for all $c < c^\ast$.
\end{lemma}

\begin{proof}
Let $c < c^\ast$ and let $w_j = (u_j,v_j)$ be a \PS{c} sequence. Then
\begin{multline} \label{203}
E(w_j) = \int_\Omega \left(\frac{1}{p}\, |\nabla u_j|^p + \frac{1}{q}\, |\nabla v_j|^q\right) dx - \frac{\lambda}{a + b} \int_\Omega |u_j|^a\, |v_j|^b\, dx\\[7.5pt]
- \int_\Omega \left(\frac{1}{p^\ast}\, |u_j|^{p^\ast} + \frac{1}{q^\ast}\, |v_j|^{q^\ast}\right) dx = c + \o(1),
\end{multline}
and applying $E'(w_j)$ to $(u_j,0)$ and $(0,v_j)$ gives
\begin{equation} \label{204}
\int_\Omega |\nabla u_j|^p\, dx - \frac{\lambda a}{a + b} \int_\Omega |u_j|^a\, |v_j|^b\, dx - \int_\Omega |u_j|^{p^\ast} dx = \o(1) \norm[W^{1,\,p}_0(\Omega)]{u_j}
\end{equation}
and
\begin{equation} \label{205}
\int_\Omega |\nabla v_j|^q\, dx - \frac{\lambda b}{a + b} \int_\Omega |u_j|^a\, |v_j|^b\, dx - \int_\Omega |v_j|^{q^\ast} dx = \o(1) \norm[W^{1,\,q}_0(\Omega)]{v_j},
\end{equation}
respectively. Taking any $p_1 \in (p,p^\ast)$ and $q_1 \in (q,q^\ast)$, dividing \eqref{204} and \eqref{205} by $p_1$ and $q_1$, respectively, and subtracting from \eqref{203} gives
\begin{multline*}
\left(\frac{1}{p} - \frac{1}{p_1}\right) \int_\Omega |\nabla u_j|^p\, dx + \left(\frac{1}{q} - \frac{1}{q_1}\right) \int_\Omega |\nabla v_j|^q\, dx - \frac{\lambda}{a + b} \left(1 - \frac{a}{p_1} - \frac{b}{q_1}\right) \int_\Omega |u_j|^a\, |v_j|^b\, dx\\[7.5pt]
+ \left(\frac{1}{p_1} - \frac{1}{p^\ast}\right) \int_\Omega |u_j|^{p^\ast} dx + \left(\frac{1}{q_1} - \frac{1}{q^\ast}\right) \int_\Omega |v_j|^{q^\ast} dx = c + \o(1) \big(1 + \norm{w_j}\big).
\end{multline*}
Since $p_0 < p^\ast$ and $q_0 < q^\ast$, this together with \eqref{103} implies that the sequence $\seq{w_j}$ is bounded in $W$.

Passing to a subsequence, we may now assume that $\seq{w_j}$ converges to some $w = (u,v)$ weakly in $W$, strongly in $L^{p_2}(\Omega) \times L^{q_2}(\Omega)$ for all $p_2 \in [1,p^\ast)$ and $q_2 \in [1,q^\ast)$, and a.e.\! in $\Omega$. Since
\[
|u_j|^a\, |v_j|^b \le \frac{a}{p_0}\, |u_j|^{p_0} + \frac{b}{q_0}\, |v_j|^{q_0}
\]
by \eqref{102} and the Young's inequality, and $w_j \to w$ in $L^{p_0}(\Omega) \times L^{q_0}(\Omega)$,
\[
\int_\Omega |u_j|^a\, |v_j|^b\, dx \to \int_\Omega |u|^a\, |v|^b\, dx.
\]
So \eqref{203} reduces to
\begin{multline} \label{206}
\int_\Omega \left(\frac{1}{p}\, |\nabla u_j|^p + \frac{1}{q}\, |\nabla v_j|^q\right) dx - \frac{\lambda}{a + b} \int_\Omega |u|^a\, |v|^b\, dx\\[7.5pt]
- \int_\Omega \left(\frac{1}{p^\ast}\, |u_j|^{p^\ast} + \frac{1}{q^\ast}\, |v_j|^{q^\ast}\right) dx = c + \o(1),
\end{multline}
and \eqref{204} and \eqref{205} reduce to
\begin{equation} \label{207}
\int_\Omega |\nabla u_j|^p\, dx - \frac{\lambda a}{a + b} \int_\Omega |u|^a\, |v|^b\, dx - \int_\Omega |u_j|^{p^\ast} dx = \o(1)
\end{equation}
and
\begin{equation} \label{208}
\int_\Omega |\nabla v_j|^q\, dx - \frac{\lambda b}{a + b} \int_\Omega |u|^a\, |v|^b\, dx - \int_\Omega |v_j|^{q^\ast} dx = \o(1),
\end{equation}
respectively. Moreover, applying $E'(w_j)$ to $(u,0)$ and $(0,v)$, and passing to the limit gives
\begin{equation} \label{209}
\int_\Omega |\nabla u|^p\, dx - \frac{\lambda a}{a + b} \int_\Omega |u|^a\, |v|^b\, dx - \int_\Omega |u|^{p^\ast} dx = 0
\end{equation}
and
\begin{equation} \label{210}
\int_\Omega |\nabla v|^q\, dx - \frac{\lambda b}{a + b} \int_\Omega |u|^a\, |v|^b\, dx - \int_\Omega |v|^{q^\ast} dx = 0,
\end{equation}
respectively.

Set
\[
\widetilde{u}_j = u_j - u, \qquad \widetilde{v}_j = v_j - v, \qquad \widetilde{w}_j = (\widetilde{u}_j,\widetilde{v}_j).
\]
We will show that $\norm{\widetilde{w}_j} \to 0$ for a subsequence. Suppose that this is not the case. By the Br{\'e}zis-Lieb lemma \cite[Theorem 1]{MR699419},
\begin{gather*}
\int_\Omega |\nabla u_j|^p\, dx - \int_\Omega |\nabla u|^p\, dx = \int_\Omega |\nabla \widetilde{u}_j|^p\, dx + \o(1),\\[7.5pt]
\int_\Omega |\nabla v_j|^p\, dx - \int_\Omega |\nabla v|^p\, dx = \int_\Omega |\nabla \widetilde{v}_j|^p\, dx + \o(1),\\[7.5pt]
\int_\Omega |u_j|^{p^\ast} dx - \int_\Omega |u|^{p^\ast} dx = \int_\Omega |\widetilde{u}_j|^{p^\ast} dx + \o(1),\\[7.5pt]
\int_\Omega |v_j|^{p^\ast} dx - \int_\Omega |v|^{p^\ast} dx = \int_\Omega |\widetilde{v}_j|^{p^\ast} dx + \o(1).
\end{gather*}
So subtracting \eqref{209} from \eqref{207} and \eqref{210} from \eqref{208}, and combining with \eqref{201} gives
\begin{equation} \label{211}
\int_\Omega |\nabla \widetilde{u}_j|^p\, dx = \int_\Omega |\widetilde{u}_j|^{p^\ast} dx + \o(1) \le \frac{1}{S_p^{p^\ast/p}} \left(\int_\Omega |\nabla \widetilde{u}_j|^p\, dx\right)^{p^\ast/p} + \o(1)
\end{equation}
and
\begin{equation} \label{212}
\int_\Omega |\nabla \widetilde{v}_j|^q\, dx = \int_\Omega |\widetilde{v}_j|^{q^\ast} dx + \o(1) \le \frac{1}{S_q^{q^\ast/q}} \left(\int_\Omega |\nabla \widetilde{v}_j|^q\, dx\right)^{q^\ast/q} + \o(1),
\end{equation}
respectively. Since $\norm{w_j}$ is bounded away from zero, either $\norm[W^{1,\,p}_0(\Omega)]{\widetilde{u}_j}$ or $\norm[W^{1,\,q}_0(\Omega)]{\widetilde{v}_j}$ is bounded away from zero for a renamed subsequence. So it follows from \eqref{211} and \eqref{212} that either
\[
\int_\Omega |\nabla \widetilde{u}_j|^p\, dx \ge S_p^{N/p} + \o(1)
\]
or
\[
\int_\Omega |\nabla \widetilde{v}_j|^q\, dx \ge S_q^{N/q} + \o(1),
\]
and hence
\begin{equation} \label{213}
\int_\Omega (|\nabla \widetilde{u}_j|^p + |\nabla \widetilde{v}_j|^q)\, dx \ge \min \set{S_p^{N/p},S_q^{N/q}} + \o(1).
\end{equation}
Dividing \eqref{207} and \eqref{208} by $p^\ast$ and $q^\ast$, respectively, and subtracting from \eqref{206} gives
\[
\frac{1}{N} \int_\Omega (|\nabla u_j|^p + |\nabla v_j|^q)\, dx - \frac{\lambda}{a + b} \left(1 - \frac{a}{p^\ast} - \frac{b}{q^\ast}\right) \int_\Omega |u|^a\, |v|^b\, dx = c + \o(1),
\]
and subtracting from this \eqref{209} and \eqref{210} divided by $N$ gives
\begin{multline*}
\frac{1}{N} \int_\Omega (|\nabla \widetilde{u}_j|^p + |\nabla \widetilde{v}_j|^q)\, dx - \frac{\lambda}{a + b} \left(1 - \frac{a}{p} - \frac{b}{q}\right) \int_\Omega |u|^a\, |v|^b\, dx\\[7.5pt]
+ \frac{1}{N} \int_\Omega \left(|u|^{p^\ast} + |v|^{q^\ast}\right) dx = c + \o(1).
\end{multline*}
Combining this with \eqref{213}, \eqref{202}, and \eqref{104} gives
\begin{equation} \label{214}
c \ge c^\ast + \frac{\lambda\, (\gamma - 1)}{a + b} \int_\Omega |u|^a\, |v|^b\, dx.
\end{equation}
Since $\gamma \ge 1$, this gives $c \ge c^\ast$, contrary to assumption.
\end{proof}

\section{Proofs}

\subsection{Proof of Theorem \ref{Theorem 101}}

We apply Theorem \ref{Theorem 305} with
\[
W = W^{1,\,p}_0(\Omega) \times W^{1,\,q}_0(\Omega) = \set{w = (u,v) : u \in W^{1,\,p}_0(\Omega),\, v \in W^{1,\,q}_0(\Omega)},
\]
$\mu = 1/p$ and $\nu = 1/q$,
\[
I(w) = \int_\Omega \left(\frac{1}{p}\, |\nabla u|^p + \frac{1}{q}\, |\nabla v|^q\right) dx, \quad w \in W,
\]
and
\[
E(w) = I(w) - \lambda J(w) - \int_\Omega \left(\frac{1}{p^\ast}\, |u|^{p^\ast} + \frac{1}{q^\ast}\, |v|^{q^\ast}\right) dx, \quad w \in W,
\]
where
\[
J(w) = \frac{1}{a + b} \int_\Omega |u|^a\, |v|^b\, dx.
\]
In view of Lemma \ref{Lemma 201}, we take
\[
c^\ast = \frac{1}{N}\, \min \set{S_p^{N/p},S_q^{N/q}}.
\]

Let $\lambda \in (\lambda_k - \delta,\lambda_k)$, let
\begin{equation} \label{514}
0 < \eps < \min \set{\delta/(\lambda_k - \lambda) - 1,(\lambda_{k+m} - \lambda_{k+m-1})/\delta},
\end{equation}
and let $\eps_\lambda = \eps\, (\lambda_k - \lambda)$. Then $\lambda_{k+m-1} < \lambda_{k+m-1} + \eps_\lambda < \lambda_{k+m}$ and hence
\[
i(\M^+ \setminus \Psi_{\lambda_{k+m-1} + \eps_\lambda}) = k + m - 1
\]
by Theorem \ref{Theorem 401} \ref{Theorem 401.iii}. Since $\M^+ \setminus \Psi_{\lambda_{k+m-1} + \eps_\lambda}$ is an open symmetric subset of $\M$, then it has a compact symmetric subset $A_0$ of index $k + m - 1$ (see the proof of Proposition 3.1 in Degiovanni and Lancelotti \cite{MR2371112}). We take $B_0 = \Psi_{\lambda_k} \cup (\M \setminus \M^+)$. Then
\[
\M \setminus B_0 = \M^+ \setminus \Psi_{\lambda_k}.
\]
Either $\lambda_1 = \cdots = \lambda_k$, or $\lambda_{l-1} < \lambda_l = \cdots = \lambda_k$ for some $2 \le l \le k$. In the former case,
\[
i(\M \setminus B_0) = i(\M^+ \setminus \Psi_{\lambda_1}) = i(\emptyset) = 0 \le k - 1
\]
by Theorem \ref{Theorem 401} \ref{Theorem 401.i} and Proposition \ref{Proposition 303} $(i_1)$. In the latter case,
\[
i(\M \setminus B_0) = i(\M^+ \setminus \Psi_{\lambda_l}) = l - 1 \le k - 1
\]
by Theorem \ref{Theorem 401} \ref{Theorem 401.iii}.

Let $R > \rho > 0$ and let
\begin{gather*}
X = \set{w_t : w \in A_0,\, 0 \le t \le R},\\
A = \set{w_R : w \in A_0},\\
B = \set{w_\rho : w \in B_0}.
\end{gather*}
For $t \ge 0$,
\begin{equation} \label{501}
E(w_t) = \begin{cases}
\ds{t \left(1 - \frac{\lambda}{\Psi(w)}\right) - \frac{t^{p^\ast/p}}{p^\ast} \int_\Omega |u|^{p^\ast} dx - \frac{t^{q^\ast/q}}{q^\ast} \int_\Omega |v|^{q^\ast} dx}, & w \in \M^+\\[25pt]
\ds{t - \frac{t^{p^\ast/p}}{p^\ast} \int_\Omega |u|^{p^\ast} dx - \frac{t^{q^\ast/q}}{q^\ast} \int_\Omega |v|^{q^\ast} dx}, & w \in \M \setminus \M^+.
\end{cases}
\end{equation}
Since $J(w) = 0$ if $\int_\Omega |u|^{p^\ast} dx$ or $\int_\Omega |v|^{q^\ast} dx$ is zero, these integrals are positive on $\M^+$ and hence bounded away from zero on the compact subset $A_0$, so \eqref{501} gives
\[
E(w_t) \le t - c_1\, t^{p^\ast/p} - c_2\, t^{q^\ast/q} \quad \forall w \in A_0
\]
for some constants $c_1, c_2 > 0$. So the first inequality in \eqref{306} holds if $R$ is sufficiently large. Since $\M$ is bounded, \eqref{501} also gives
\[
E(w_t) \ge \begin{cases}
t \left(1 - \dfrac{\lambda}{\lambda_k}\right) - c_3\, t^{p^\ast/p} - c_4\, t^{q^\ast/q} & \forall w \in \Psi_{\lambda_k}\\[20pt]
t - c_3\, t^{p^\ast/p} - c_4\, t^{q^\ast/q} & \forall w \in \M \setminus \M^+
\end{cases}
\]
for some constants $c_3, c_4 > 0$, so the second inequality in \eqref{306} holds if $\lambda < \lambda_k$ and $\rho$ is sufficiently small.

Any $w \in X$ can be written as $w = \widetilde{w}_t$ for some $\widetilde{w} \in A_0$ and $t \in [0,R]$. Then
\[
I(w) = t\, I(\widetilde{w}) = t, \qquad J(w) = t\, J(\widetilde{w}) = \frac{I(w)}{\Psi(\widetilde{w})}.
\]
Since $A_0 \subset \M^+ \setminus \Psi_{\lambda_{k+m-1} + \eps_\lambda}$ and hence $\Psi(\widetilde{w}) < \lambda_{k+m-1} + \eps_\lambda = \lambda_k + \eps_\lambda$, this gives
\[
I(w) \le (\lambda_k + \eps_\lambda)\, J(w),
\]
so
\[
E(w) \le (\lambda_k + \eps_\lambda - \lambda)\, J(w) - \int_\Omega \left(\frac{1}{p^\ast}\, |u|^{p^\ast} + \frac{1}{q^\ast}\, |v|^{q^\ast}\right) dx.
\]
Since
\[
J(w) \le \frac{\vol{\Omega}^{1/\zeta}}{a + b} \left(\int_\Omega |u|^{p^\ast} dx\right)^{a/p^\ast}\! \left(\int_\Omega |v|^{q^\ast} dx\right)^{b/q^\ast}
\]
by the H\"{o}lder inequality, this in turn gives
\[
E(w) \le \frac{(1 + \eps)(\lambda_k - \lambda) \vol{\Omega}^{1/\zeta}}{a + b}\, \sigma^{a/p^\ast} \tau^{b/q^\ast} - \frac{\sigma}{p^\ast} - \frac{\tau}{q^\ast},
\]
where
\[
\sigma = \int_\Omega |u|^{p^\ast} dx, \qquad \tau = \int_\Omega |v|^{q^\ast} dx.
\]
Maximizing the right-hand side over all $\sigma, \tau \ge 0$ gives
\[
\sup_{w \in X}\, E(w) \le \frac{\vol{\Omega}}{\zeta} \left[\frac{a^{a/p^\ast} b^{b/q^\ast}}{a + b}\, (1 + \eps)(\lambda_k - \lambda)\right]^\zeta < \frac{\vol{\Omega}}{\zeta} \left(\frac{a^{a/p^\ast} b^{b/q^\ast}}{a + b}\, \delta\right)^\zeta = \frac{1}{N}\, \min \set{S_p^{N/p},S_q^{N/q}}
\]
by \eqref{514} and \eqref{515}. It now follows from Theorem \ref{Theorem 305} that system \eqref{101} has $m$ distinct pairs of solutions $\pm w^\lambda_j,\, j = 1,\dots,m$ such that
\begin{equation} \label{516}
0 < \inf_{w \in B}\, E(w) \le E(w^\lambda_j) \le \sup_{w \in X}\, E(w) \le \frac{\vol{\Omega}}{\zeta} \left[\frac{a^{a/p^\ast} b^{b/q^\ast}}{a + b}\, (1 + \eps)(\lambda_k - \lambda)\right]^\zeta
\end{equation}
for each $j$.

It remains to show that $w^\lambda_j \to 0$ as $\lambda \nearrow \lambda_k$. Let $w^\lambda_j = (u^\lambda_j,v^\lambda_j)$. By \eqref{516},
\begin{multline} \label{517}
E(w^\lambda_j) = \int_\Omega \left(\frac{1}{p}\, |\nabla u^\lambda_j|^p + \frac{1}{q}\, |\nabla v^\lambda_j|^q\right) dx - \frac{\lambda}{a + b} \int_\Omega |u^\lambda_j|^a\, |v^\lambda_j|^b\, dx\\[7.5pt]
- \int_\Omega \left(\frac{1}{p^\ast}\, |u^\lambda_j|^{p^\ast} + \frac{1}{q^\ast}\, |v^\lambda_j|^{q^\ast}\right) dx = \o(1)
\end{multline}
as $\lambda \nearrow \lambda_k$. Since $w^\lambda_j$ is a critical point of $E$,
\begin{equation} \label{518}
E'(w^\lambda_j)\, (u^\lambda_j,0) = \int_\Omega |\nabla u^\lambda_j|^p\, dx - \frac{\lambda a}{a + b} \int_\Omega |u^\lambda_j|^a\, |v^\lambda_j|^b\, dx - \int_\Omega |u^\lambda_j|^{p^\ast} dx = 0
\end{equation}
and
\begin{equation} \label{519}
E'(w^\lambda_j)\, (0,v^\lambda_j) = \int_\Omega |\nabla v^\lambda_j|^q\, dx - \frac{\lambda b}{a + b} \int_\Omega |u^\lambda_j|^a\, |v^\lambda_j|^b\, dx - \int_\Omega |v^\lambda_j|^{q^\ast} dx = 0.
\end{equation}
Dividing \eqref{518} and \eqref{519} by $p$ and $q$, respectively, and subtracting from \eqref{517} gives
\[
\frac{1}{N} \int_\Omega \left(|u^\lambda_j|^{p^\ast} + |v^\lambda_j|^{q^\ast}\right) dx = \o(1)
\]
since $\gamma = 1$, so
\[
\int_\Omega |u^\lambda_j|^{p^\ast} dx = \o(1), \qquad \int_\Omega |v^\lambda_j|^{q^\ast} dx = \o(1).
\]
Since
\[
\int_\Omega |u^\lambda_j|^a\, |v^\lambda_j|^b\, dx \le \left(\int_\Omega |u^\lambda_j|^p\, dx\right)^{a/p}\! \left(\int_\Omega |v^\lambda_j|^q\, dx\right)^{b/q}
\]
by the H\"{o}lder inequality, then
\[
\int_\Omega |u^\lambda_j|^a\, |v^\lambda_j|^b\, dx = \o(1).
\]
The desired conclusion now follows from \eqref{517}.

\subsection{Proof of Theorem \ref{Theorem 104}}

We apply Corollary \ref{Corollary 306} with
\[
W = W^{1,\,p}_0(\Omega) \times W^{1,\,q}_0(\Omega) = \set{w = (u,v) : u \in W^{1,\,p}_0(\Omega),\, v \in W^{1,\,q}_0(\Omega)},
\]
$\mu = 1/p$ and $\nu = 1/q$,
\[
I(w) = \int_\Omega \left(\frac{1}{p}\, |\nabla u|^p + \frac{1}{q}\, |\nabla v|^q\right) dx, \quad w \in W,
\]
and
\[
E(w) = I(w) - \frac{\lambda}{a + b} \int_\Omega |u|^a\, |v|^b\, dx - \int_\Omega \left(\frac{1}{p^\ast}\, |u|^{p^\ast} + \frac{1}{q^\ast}\, |v|^{q^\ast}\right) dx, \quad w \in W.
\]
In view of Lemma \ref{Lemma 201}, we take
\[
c^\ast = \frac{1}{N}\, \min \set{S_p^{N/p},S_q^{N/q}}.
\]

Let $\alpha, \beta > 1$ satisfy \eqref{106} and let $\seq{\lambda_k}$ be the sequence of eigenvalues of the eigenvalue problem \eqref{105} based on the cohomological index (see Theorem \ref{Theorem 401}). Since $\lambda_k \nearrow \infty$, by taking $m$ larger if necessary, we may assume that $\lambda_m < \lambda_{m+1}$. Then
\[
i(\M^+ \setminus \Psi_{\lambda_{m+1}}) = m
\]
by Theorem \ref{Theorem 401} \ref{Theorem 401.iii}. Since $\M^+ \setminus \Psi_{\lambda_{m+1}}$ is an open symmetric subset of $\M$, then it has a compact symmetric subset $A_0$ of index $m$ (see the proof of Proposition 3.1 in Degiovanni and Lancelotti \cite{MR2371112}).

Let $R > \rho > 0$ and let
\[
A = \set{w_R : w \in A_0}, \qquad X = \set{w_t : w \in A_0,\, 0 \le t \le R}.
\]
For $t \ge 0$,
\begin{equation} \label{502}
E(w_t) = t - \frac{\lambda t^\gamma}{a + b} \int_\Omega |u|^a\, |v|^b\, dx - \frac{t^{p^\ast/p}}{p^\ast} \int_\Omega |u|^{p^\ast} dx - \frac{t^{q^\ast/q}}{q^\ast} \int_\Omega |v|^{q^\ast} dx, \quad w \in \M.
\end{equation}
Since $J(w) = 0$ if $\int_\Omega |u|^a\, |v|^b\, dx$ is zero, this integral is positive on $\M^+$ and hence bounded away from zero on the compact subset $A_0$, so \eqref{502} gives
\begin{equation} \label{503}
E(w_t) \le t - c_1\, \lambda t^\gamma \quad \forall w \in A_0
\end{equation}
for some constant $c_1 > 0$. Since $\lambda > 0$ and $\gamma > 1$, it follows that the first inequality in \eqref{311} holds if $R$ is sufficiently large. Since $\M$ is bounded, \eqref{502} also gives
\[
E(w_t) \ge t - c_2\, \lambda t^\gamma - c_3\, t^{p^\ast/p} - c_4\, t^{q^\ast/q} \quad \forall w \in \M
\]
for some constants $c_2, c_3, c_4 > 0$, so the second inequality in \eqref{311} holds if $\rho$ is sufficiently small. Maximizing the right-hand side of \eqref{503} over all $t \ge 0$ gives
\[
\sup_{w \in X}\, E(w) \le \frac{\gamma - 1}{\gamma\, (c_1\, \lambda \gamma)^{1/(\gamma - 1)}},
\]
so the last inequality in \eqref{311} holds if $\lambda$ is sufficiently large. It now follows from Corollary \ref{Corollary 306} that system \eqref{101} has $m$ distinct pairs of solutions $\pm w^\lambda_j,\, j = 1,\dots,m$ such that
\begin{equation} \label{520}
0 < \inf_{w \in \M_\rho}\, E(w) \le E(w^\lambda_j) \le \sup_{w \in X}\, E(w) \le \frac{\gamma - 1}{\gamma\, (c_1\, \lambda \gamma)^{1/(\gamma - 1)}}
\end{equation}
for each $j$.

It remains to show that $w^\lambda_j \to 0$ as $\lambda \to \infty$. Let $w^\lambda_j = (u^\lambda_j,v^\lambda_j)$. By \eqref{520},
\begin{multline} \label{521}
E(w^\lambda_j) = \int_\Omega \left(\frac{1}{p}\, |\nabla u^\lambda_j|^p + \frac{1}{q}\, |\nabla v^\lambda_j|^q\right) dx - \frac{\lambda}{a + b} \int_\Omega |u^\lambda_j|^a\, |v^\lambda_j|^b\, dx\\[7.5pt]
- \int_\Omega \left(\frac{1}{p^\ast}\, |u^\lambda_j|^{p^\ast} + \frac{1}{q^\ast}\, |v^\lambda_j|^{q^\ast}\right) dx = \o(1)
\end{multline}
as $\lambda \to \infty$. Since $w^\lambda_j$ is a critical point of $E$,
\begin{equation} \label{522}
E'(w^\lambda_j)\, (u^\lambda_j,0) = \int_\Omega |\nabla u^\lambda_j|^p\, dx - \frac{\lambda a}{a + b} \int_\Omega |u^\lambda_j|^a\, |v^\lambda_j|^b\, dx - \int_\Omega |u^\lambda_j|^{p^\ast} dx = 0
\end{equation}
and
\begin{equation} \label{523}
E'(w^\lambda_j)\, (0,v^\lambda_j) = \int_\Omega |\nabla v^\lambda_j|^q\, dx - \frac{\lambda b}{a + b} \int_\Omega |u^\lambda_j|^a\, |v^\lambda_j|^b\, dx - \int_\Omega |v^\lambda_j|^{q^\ast} dx = 0.
\end{equation}
Dividing \eqref{522} and \eqref{523} by $p$ and $q$, respectively, and subtracting from \eqref{521} gives
\[
\frac{\lambda\, (\gamma - 1)}{a + b} \int_\Omega |u^\lambda_j|^a\, |v^\lambda_j|^b\, dx + \frac{1}{N} \int_\Omega \left(|u^\lambda_j|^{p^\ast} + |v^\lambda_j|^{q^\ast}\right) dx = \o(1)
\]
in view of \eqref{104}. Since $\lambda > 0$ and $\gamma > 1$, this implies that
\[
\int_\Omega |u^\lambda_j|^a\, |v^\lambda_j|^b\, dx = \o(1), \qquad \int_\Omega |u^\lambda_j|^{p^\ast} dx = \o(1), \qquad \int_\Omega |v^\lambda_j|^{q^\ast} dx = \o(1).
\]
The desired conclusion now follows from \eqref{521}.

\subsection{Proof of Theorem \ref{Theorem 106}}

We have
\[
E(w) = I(w) - \frac{\lambda}{a + b} \int_\Omega |u|^a\, |v|^b\, dx - \int_\Omega \left(\frac{1}{p^\ast}\, |u|^{p^\ast} + \frac{1}{q^\ast}\, |v|^{q^\ast}\right) dx, \quad w \in W,
\]
where
\[
I(w) = \int_\Omega \left(\frac{1}{p}\, |\nabla u|^p + \frac{1}{q}\, |\nabla v|^q\right) dx.
\]
We will use a truncation of $E$ inspired by Garc{\'{\i}}a Azorero and Peral Alonso \cite{MR1083144}.

By the H\"{o}lder inequality,
\[
\frac{1}{a + b} \int_\Omega |u|^a\, |v|^b\, dx \le \frac{\vol{\Omega}^{1 - \gamma}}{a + b} \left(\int_\Omega |u|^p\, dx\right)^{a/p}\! \left(\int_\Omega |v|^q\, dx\right)^{b/q} \le c_1\, I(w)^\gamma
\]
for some constant $c_1 > 0$. By \eqref{201},
\begin{multline*}
\int_\Omega \left(\frac{1}{p^\ast}\, |u|^{p^\ast} + \frac{1}{q^\ast}\, |v|^{q^\ast}\right) dx \le \frac{S_p^{- p^\ast/p}}{p^\ast} \left(\int_\Omega |\nabla u|^p\, dx\right)^{p^\ast/p} + \frac{S_q^{- q^\ast/q}}{q^\ast} \left(\int_\Omega |\nabla v|^q\, dx\right)^{q^\ast/q}\\[7.5pt]
\le c_2\, I(w)^{p^\ast/p} + c_3\, I(w)^{q^\ast/q}
\end{multline*}
for some constants $c_2, c_3 > 0$. So
\[
E(w) \ge I(w) - c_1\, \lambda I(w)^\gamma - c_2\, I(w)^{p^\ast/p} - c_3\, I(w)^{q^\ast/q} = g_\lambda(I(w)),
\]
where
\[
g_\lambda(t) = t - c_1\, \lambda t^\gamma - c_2\, t^{p^\ast/p} - c_3\, t^{q^\ast/q}, \quad t \ge 0.
\]
Since $\gamma < 1$, $\exists \lambda_\ast > 0$ such that for all $\lambda \in (0,\lambda_\ast)$, there are $0 < R_1(\lambda) < R_2(\lambda)$ such that
\[
g_\lambda(t) < 0 \quad \forall t \in (0,R_1(\lambda)) \cup (R_2(\lambda),\infty), \qquad g_\lambda(t) > 0 \quad \forall t \in (R_1(\lambda),R_2(\lambda)).
\]
Since $g_\lambda \ge g_{\lambda_\ast}$,
\begin{equation} \label{504}
R_1(\lambda) \le R_1(\lambda_\ast), \qquad R_2(\lambda) \ge R_2(\lambda_\ast).
\end{equation}

Take a smooth function $\xi_\lambda : [0,\infty) \to [0,1]$ such that $\xi_\lambda(t) = 1$ for $t \in [0,R_1(\lambda)]$ and $\xi_\lambda(t) = 0$ for $t \in [R_2(\lambda),\infty)$, and set
\[
\widetilde{E}(w) = \xi_\lambda(I(w))\, E(w).
\]
If $\widetilde{E}(w) < 0$, then $I(w) \in (0,R_1(\lambda))$ and hence $I(z) \in (0,R_1(\lambda))$ for all $z$ in some neighborhood of $w$ by continuity. Then $\widetilde{E}$ coincides with $E$ in a neighborhood of $w$, so
\[
\widetilde{E}(w) = E(w), \qquad \widetilde{E}'(w) = E'(w).
\]
In particular, critical points of $\widetilde{E}$ at negative levels are also critical points of $E$ at the same levels.

First we show that $\widetilde{E}$ satisfies the \PS{c} condition for all $c < 0$ if $\lambda \in (0,\lambda_\ast)$ is sufficiently small. Let $c < 0$ and let $\seq{w_j} \subset W$ be a \PS{c} sequence. Since $\widetilde{E}(w_j) \to c$, for all sufficiently large $j$, $\widetilde{E}(w_j) < 0$ and hence $\widetilde{E}(w_j) = E(w_j)$ and $\widetilde{E}'(w_j) = E'(w_j)$. So $\seq{w_j}$ has a renamed subsequence for which $E(w_j) \to c$ and $E'(w_j) \to 0$. Moreover, $I(w_j) \in (0,R_1(\lambda))$, which together with the first inequality in \eqref{504} implies that $\seq{w_j}$ is bounded independently of $\lambda \in (0,\lambda_\ast)$. We may now proceed as in the proof of Lemma \ref{Lemma 201} to arrive at the inequality
\[
c \ge \frac{1}{N}\, \min \set{S_p^{N/p},S_q^{N/q}} - \frac{\lambda\, (1 - \gamma)}{a + b} \int_\Omega |u|^a\, |v|^b\, dx
\]
if $\seq{w_j}$ does not converge to its weak limit $w = (u,v)$ (see \eqref{214}). Since $c < 0$, this gives
\[
\frac{\lambda\, (1 - \gamma)}{a + b} \int_\Omega |u|^a\, |v|^b\, dx > \frac{1}{N}\, \min \set{S_p^{N/p},S_q^{N/q}}.
\]
However, since $\seq{w_j}$ is bounded independently of $\lambda$, so is $w$, and hence this inequality cannot hold for sufficiently small $\lambda$.

By taking $\lambda_\ast$ smaller if necessary, we may now assume that $\widetilde{E}$ satisfies the \PS{c} condition for all $c < 0$ for $\lambda \in (0,\lambda_\ast)$. Let $\F$ denote the class of symmetric subsets of $W \setminus \set{0}$. For $k \ge 1$, let $\F_k = \bgset{M \in \F : i(M) \ge k}$ and set
\[
c_k := \inf_{M \in \F_k}\, \sup_{w \in M}\, \widetilde{E}(w).
\]
Since $\widetilde{E}$ is clearly bounded from below, $c_k > - \infty$. We will show that $c_k < 0$. It will then follow from a standard argument in critical point theory that $c_k \nearrow 0$ is a sequence of critical values of $\widetilde{E}$ and hence also of $E$ (see, e.g., Perera et al.\! \cite[Proposition 3.36]{MR2640827}).

Since $\F_k \supset \F_{k+1}$ and hence $c_k \le c_{k+1}$, it suffices to show that $c_k < 0$ for arbitrarily large $k$. Let $\alpha, \beta > 1$ satisfy \eqref{106} and let $\seq{\lambda_k}$ be the sequence of eigenvalues of the eigenvalue problem \eqref{105} based on the cohomological index (see Theorem \ref{Theorem 401}). Since $\lambda_k \nearrow \infty$, there are arbitrarily large $k$ for which $\lambda_k < \lambda_{k+1}$. For such $k$,
\[
i(\M^+ \setminus \Psi_{\lambda_{k+1}}) = k
\]
by Theorem \ref{Theorem 401} \ref{Theorem 401.iii}. Since $\M^+ \setminus \Psi_{\lambda_{k+1}}$ is an open symmetric subset of $\M$, then it has a compact symmetric subset $M$ of index $k$ (see the proof of Proposition 3.1 in Degiovanni and Lancelotti \cite{MR2371112}).

For $t > 0$, let $M_t = \bgset{w_t : w \in M}$. Since the mapping $M \to M_t,\, w \mapsto w_t$ is an odd homeomorphism,
\[
i(M_t) = i(M) = k
\]
by Proposition \ref{Proposition 303} $(i_2)$, so $M_t \in \F_k$. Since $I(w_t) = t$ for $w \in M$, if $t \le R_1(\lambda)$, then $\widetilde{E} = E$ on $M_t$ and hence
\begin{equation} \label{505}
c_k \le \sup_{w \in M_t}\, \widetilde{E}(w) = \sup_{w \in M_t}\, E(w) = \sup_{w \in M}\, E(w_t).
\end{equation}
For $w \in M$,
\begin{equation} \label{506}
E(w_t) = t - \frac{\lambda t^\gamma}{a + b} \int_\Omega |u|^a\, |v|^b\, dx - \frac{t^{p^\ast/p}}{p^\ast} \int_\Omega |u|^{p^\ast} dx - \frac{t^{q^\ast/q}}{q^\ast} \int_\Omega |v|^{q^\ast} dx.
\end{equation}
Since $J(w) = 0$ if $\int_\Omega |u|^a\, |v|^b\, dx$ is zero, this integral is positive on $\M^+$ and hence bounded away from zero on the compact subset $M$, so \eqref{506} gives
\[
E(w_t) \le t - c_4\, \lambda t^\gamma
\]
for some constant $c_4 > 0$. Since $\lambda > 0$ and $\gamma < 1$, the right-hand side is negative if $t$ is sufficiently small. So $c_k < 0$ by \eqref{505}.

Let $w^\lambda_k = (u^\lambda_k,v^\lambda_k)$ be a critical point of $\widetilde{E}$ at the level $c_k$. It remains to show that $w^\lambda_k \to 0$ as $\lambda \searrow 0$. Since
\[
\widetilde{E}(w^\lambda_k) = c_k < 0,
\]
we have
\[
I(w^\lambda_k) = \int_\Omega \left(\frac{1}{p}\, |\nabla u^\lambda_k|^p + \frac{1}{q}\, |\nabla v^\lambda_k|^q\right) dx < R_1(\lambda),
\]
so it suffices to show that $R_1(\lambda) \to 0$ as $\lambda \searrow 0$. Since
\[
g_\lambda(\lambda) = \lambda - c_1\, \lambda^{1 + \gamma} - c_2\, \lambda^{p^\ast/p} - c_3\, \lambda^{q^\ast/q} \ge 0
\]
if $\lambda$ is sufficiently small, and $g_\lambda(t) < 0$ for all $t \in (0,R_1(\lambda))$, we have $0 < R_1(\lambda) \le \lambda$ for all sufficiently small $\lambda$, from which the desired conclusion follows.

\subsection{Proof of Theorem \ref{Theorem 110}}

We apply Corollary \ref{Corollary 306} with
\[
W = W^{1,\,p}_0(\Omega) \times W^{1,\,q}_0(\Omega) = \set{w = (u,v) : u \in W^{1,\,p}_0(\Omega),\, v \in W^{1,\,q}_0(\Omega)},
\]
$\mu = 1/p$ and $\nu = 1/q$,
\[
I(w) = \int_\Omega \left(\frac{1}{p}\, |\nabla u|^p + \frac{1}{q}\, |\nabla v|^q\right) dx, \quad w \in W,
\]
and
\[
E(w) = I(w) - \lambda J(w) + \mu K(w) - \int_\Omega \left(\frac{1}{p^\ast}\, |u|^{p^\ast} + \frac{1}{q^\ast}\, |v|^{q^\ast}\right) dx, \quad w \in W,
\]
where
\[
J(w) = \frac{1}{a + b} \int_\Omega |u|^a\, |v|^b\, dx, \qquad K(w) = \frac{1}{c + d} \int_\Omega |u|^c\, |v|^d\, dx.
\]

First we show that $E$ satisfies the \PS{c} condition for all $c < c^\ast$, where $c^\ast$ is as in \eqref{202}. Let $c < c^\ast$ and let $\seq{w_j} \subset W$ be a \PS{c} sequence. As in the proof of Lemma \ref{Lemma 201}, $\seq{w_j}$ is bounded and
\[
c \ge c^\ast + \frac{\mu\, (1 - \kappa)}{c + d} \int_\Omega |u|^c\, |v|^d\, dx
\]
if $\seq{w_j}$ does not converge to its weak limit $w = (u,v)$ (see \eqref{214}). This together with \eqref{111} gives $c \ge c^\ast$, contrary to assumption.

By taking $k$ larger if necessary, we may assume that $\lambda_k < \lambda_{k+1}$. Fix $\underline{\lambda} \le \lambda$ such that $\lambda_k < \underline{\lambda} < \lambda_{k+1}$. Then
\[
i(\M^+ \setminus \Psi_{\underline{\lambda}}) = k
\]
by Theorem \ref{Theorem 401} \ref{Theorem 401.iii}. Since $\M^+ \setminus \Psi_{\underline{\lambda}}$ is an open symmetric subset of $\M$, then it has a compact symmetric subset $A_0$ of index $k$ (see the proof of Proposition 3.1 in Degiovanni and Lancelotti \cite{MR2371112}).

Let $R > \rho > 0$ and let
\[
A = \set{w_R : w \in A_0}, \qquad X = \set{w_t : w \in A_0,\, 0 \le t \le R}.
\]
For $t \ge 0$,
\begin{equation} \label{507}
E(w_t) = \begin{cases}
\ds{t \left(1 - \frac{\lambda}{\Psi(w)}\right) + \mu t^\kappa\, K(w) - \frac{t^{p^\ast/p}}{p^\ast} \int_\Omega |u|^{p^\ast} dx - \frac{t^{q^\ast/q}}{q^\ast} \int_\Omega |v|^{q^\ast} dx}, & w \in \M^+\\[25pt]
\ds{t + \mu t^\kappa\, K(w) - \frac{t^{p^\ast/p}}{p^\ast} \int_\Omega |u|^{p^\ast} dx - \frac{t^{q^\ast/q}}{q^\ast} \int_\Omega |v|^{q^\ast} dx}, & w \in \M \setminus \M^+.
\end{cases}
\end{equation}
Since $J(w) = 0$ if $\int_\Omega |u|^{p^\ast} dx$ or $\int_\Omega |v|^{q^\ast} dx$ is zero, these integrals are positive on $\M^+$ and hence bounded away from zero on the compact subset $A_0$, so \eqref{507} gives
\[
E(w_t) \le t + c_1\, t^\kappa - c_2\, t^{p^\ast/p} - c_3\, t^{q^\ast/q} \quad \forall w \in A_0
\]
for some constants $c_1, c_2, c_3 > 0$. So the first inequality in \eqref{311} holds if $R$ is sufficiently large. Fix $\overline{\lambda} > \lambda$. Since $\M$ is bounded, \eqref{507} together with Theorem \ref{Theorem 401} \ref{Theorem 401.i} gives
\begin{equation} \label{508}
E(w_t) \ge \begin{cases}
t \left(1 - \dfrac{\lambda}{\overline{\lambda}}\right) - c_4\, t^{p^\ast/p} - c_5\, t^{q^\ast/q} & \forall w \in \Psi_{\overline{\lambda}}\\[20pt]
\mu t^\kappa\, K(w) - t \left(\dfrac{\lambda}{\lambda_1} - 1\right) - c_4\, t^{p^\ast/p} - c_5\, t^{q^\ast/q} & \forall w \in \M^+ \setminus \Psi_{\overline{\lambda}}\\[25pt]
t - c_4\, t^{p^\ast/p} - c_5\, t^{q^\ast/q} & \forall w \in \M \setminus \M^+
\end{cases}
\end{equation}
for some constants $c_4, c_5 > 0$. We will show that
\begin{equation} \label{509}
\inf_{w \in \M^+ \setminus \Psi_{\overline{\lambda}}}\, K(w) > 0.
\end{equation}
Since $\mu > 0$ and $\kappa < 1$, it will then follow from \eqref{508} that the second inequality in \eqref{311} holds if $\rho$ is sufficiently small.

Let $\theta \in (0,\min \set{1,a/c,b/d,\eta/\zeta})$, where $\zeta$ and $\eta$ are as in \eqref{113} and \eqref{114}, respectively. Writing
\[
|u|^a\, |v|^b = \left(|u|^c\, |v|^d\right)^\theta |u|^{a - \theta c}\, |v|^{b - \theta d}
\]
and noting that
\[
1 - \theta - \frac{a - \theta c}{p^\ast} - \frac{b - \theta d}{q^\ast} = \frac{1}{\zeta} - \frac{\theta}{\eta},
\]
the H\"{o}lder inequality gives
\[
\int_\Omega |u|^a\, |v|^b\, dx \le \vol{\Omega}^{1/\zeta - \theta/\eta} \left(\int_\Omega |u|^c\, |v|^d\, dx\right)^\theta \left(\int_\Omega |u|^{p^\ast} dx\right)^{(a - \theta c)/p^\ast} \left(\int_\Omega |v|^{q^\ast} dx\right)^{(b - \theta d)/q^\ast}.
\]
Since
\[
\int_\Omega |u|^a\, |v|^b\, dx = (a + b)\, J(w) = \frac{a + b}{\Psi(w)} > \frac{a + b}{\overline{\lambda}} \quad \forall w \in \M^+ \setminus \Psi_{\overline{\lambda}}
\]
and $\int_\Omega |u|^{p^\ast} dx$ and $\int_\Omega |v|^{q^\ast} dx$ are bounded on $\M^+ \setminus \Psi_{\overline{\lambda}}$, \eqref{509} follows.

Any $w \in X$ can be written as $w = \widetilde{w}_t$ for some $\widetilde{w} \in A_0$ and $t \in [0,R]$. Then
\[
I(w) = t\, I(\widetilde{w}) = t, \qquad J(w) = t\, J(\widetilde{w}) = \frac{I(w)}{\Psi(\widetilde{w})}.
\]
Since $A_0 \subset \M^+ \setminus \Psi_{\underline{\lambda}}$ and hence $\Psi(\widetilde{w}) < \underline{\lambda} \le \lambda$, this gives $I(w) \le \lambda J(w)$, so
\[
E(w) \le \mu K(w) - \int_\Omega \left(\frac{1}{p^\ast}\, |u|^{p^\ast} + \frac{1}{q^\ast}\, |v|^{q^\ast}\right) dx.
\]
Since
\[
K(w) \le \frac{\vol{\Omega}^{1/\eta}}{c + d} \left(\int_\Omega |u|^{p^\ast} dx\right)^{c/p^\ast}\! \left(\int_\Omega |v|^{q^\ast} dx\right)^{d/q^\ast}
\]
by the H\"{o}lder inequality, this in turn gives
\[
E(w) \le \frac{\mu \vol{\Omega}^{1/\eta}}{c + d}\, \sigma^{c/p^\ast} \tau^{d/q^\ast} - \frac{\sigma}{p^\ast} - \frac{\tau}{q^\ast},
\]
where
\[
\sigma = \int_\Omega |u|^{p^\ast} dx, \qquad \tau = \int_\Omega |v|^{q^\ast} dx.
\]
Maximizing the right-hand side over all $\sigma, \tau \ge 0$ gives
\[
\sup_{w \in X}\, E(w) \le \frac{\vol{\Omega}}{\eta} \left(\frac{c^{c/p^\ast} d^{d/q^\ast}}{c + d}\, \mu\right)^\eta < \frac{1}{N}\, \min \set{S_p^{N/p},S_q^{N/q}}
\]
by \eqref{112}. It now follows from Corollary \ref{Corollary 306} that system \eqref{109} has $k$ distinct pairs of solutions $\pm w^{\lambda,\,\mu}_j,\, j = 1,\dots,k$ such that
\begin{equation} \label{510}
0 < \inf_{w \in \M_\rho}\, E(w) \le E(w^{\lambda,\,\mu}_j) \le \sup_{w \in X}\, E(w) \le \frac{\vol{\Omega}}{\eta} \left(\frac{c^{c/p^\ast} d^{d/q^\ast}}{c + d}\, \mu\right)^\eta
\end{equation}
for each $j$.

It remains to show that $w^{\lambda,\,\mu}_j \to 0$ as $\mu \searrow 0$. Let $w^{\lambda,\,\mu}_j = (u^{\lambda,\,\mu}_j,v^{\lambda,\,\mu}_j)$. By \eqref{510},
\begin{multline} \label{511}
E(w^{\lambda,\,\mu}_j) = \int_\Omega \left(\frac{1}{p}\, |\nabla u^{\lambda,\,\mu}_j|^p + \frac{1}{q}\, |\nabla v^{\lambda,\,\mu}_j|^q\right) dx - \frac{\lambda}{a + b} \int_\Omega |u^{\lambda,\,\mu}_j|^a\, |v^{\lambda,\,\mu}_j|^b\, dx\\[7.5pt]
+ \frac{\mu}{c + d} \int_\Omega |u^{\lambda,\,\mu}_j|^c\, |v^{\lambda,\,\mu}_j|^d\, dx - \int_\Omega \left(\frac{1}{p^\ast}\, |u^{\lambda,\,\mu}_j|^{p^\ast} + \frac{1}{q^\ast}\, |v^{\lambda,\,\mu}_j|^{q^\ast}\right) dx = \o(1)
\end{multline}
as $\mu \searrow 0$. Since $w^{\lambda,\,\mu}_j$ is a critical point of $E$,
\begin{multline} \label{512}
E'(w^{\lambda,\,\mu}_j)\, (u^{\lambda,\,\mu}_j,0) = \int_\Omega |\nabla u^{\lambda,\,\mu}_j|^p\, dx - \frac{\lambda a}{a + b} \int_\Omega |u^{\lambda,\,\mu}_j|^a\, |v^{\lambda,\,\mu}_j|^b\, dx\\[7.5pt]
+ \frac{\mu c}{c + d} \int_\Omega |u^{\lambda,\,\mu}_j|^c\, |v^{\lambda,\,\mu}_j|^d\, dx - \int_\Omega |u^{\lambda,\,\mu}_j|^{p^\ast} dx = 0
\end{multline}
and
\begin{multline} \label{513}
E'(w^{\lambda,\,\mu}_j)\, (0,v^{\lambda,\,\mu}_j) = \int_\Omega |\nabla v^{\lambda,\,\mu}_j|^q\, dx - \frac{\lambda b}{a + b} \int_\Omega |u^{\lambda,\,\mu}_j|^a\, |v^{\lambda,\,\mu}_j|^b\, dx\\[7.5pt]
+ \frac{\mu d}{c + d} \int_\Omega |u^{\lambda,\,\mu}_j|^c\, |v^{\lambda,\,\mu}_j|^d\, dx - \int_\Omega |v^{\lambda,\,\mu}_j|^{q^\ast} dx = 0.
\end{multline}
Dividing \eqref{512} and \eqref{513} by $p$ and $q$, respectively, and subtracting from \eqref{511} gives
\[
\frac{\mu\, (1 - \kappa)}{c + d} \int_\Omega |u^{\lambda,\,\mu}_j|^c\, |v^{\lambda,\,\mu}_j|^d\, dx + \frac{1}{N} \int_\Omega \left(|u^{\lambda,\,\mu}_j|^{p^\ast} + |v^{\lambda,\,\mu}_j|^{q^\ast}\right) dx = \o(1)
\]
in view of \eqref{110} and \eqref{111}. Since $\mu > 0$ and $\kappa < 1$, this implies that
\[
\int_\Omega |u^{\lambda,\,\mu}_j|^c\, |v^{\lambda,\,\mu}_j|^d\, dx = \o(1), \qquad \int_\Omega |u^{\lambda,\,\mu}_j|^{p^\ast} dx = \o(1), \qquad \int_\Omega |v^{\lambda,\,\mu}_j|^{q^\ast} dx = \o(1).
\]
Since
\[
\int_\Omega |u^{\lambda,\,\mu}_j|^a\, |v^{\lambda,\,\mu}_j|^b\, dx \le \left(\int_\Omega |u^{\lambda,\,\mu}_j|^p\, dx\right)^{a/p}\! \left(\int_\Omega |v^{\lambda,\,\mu}_j|^q\, dx\right)^{b/q}
\]
by the H\"{o}lder inequality, then
\[
\int_\Omega |u^{\lambda,\,\mu}_j|^a\, |v^{\lambda,\,\mu}_j|^b\, dx = \o(1).
\]
The desired conclusion now follows from \eqref{511}.

\def\cprime{$''$}

Data availability statement: The manuscript has no associated data.


\begin{thebibliography}{10}

\bibitem{MR84c:58014}
Vieri Benci.
\newblock On critical point theory for indefinite functionals in the presence
  of symmetries.
\newblock {\em Trans. Amer. Math. Soc.}, 274(2):533--572, 1982.

\bibitem{MR1970963}
Yuri Bozhkov and Enzo Mitidieri.
\newblock Existence of multiple solutions for quasilinear systems via fibering
  method.
\newblock {\em J. Differential Equations}, 190(1):239--267, 2003.

\bibitem{MR699419}
Ha{\"{\i}}m Br{\'e}zis and Elliott Lieb.
\newblock A relation between pointwise convergence of functions and convergence
  of functionals.
\newblock {\em Proc. Amer. Math. Soc.}, 88(3):486--490, 1983.

\bibitem{MR4250461}
Anna~Maria Candela, Addolorata Salvatore, and Caterina Sportelli.
\newblock Existence and multiplicity results for a class of coupled quasilinear
  elliptic systems of gradient type.
\newblock {\em Adv. Nonlinear Stud.}, 21(2):461--488, 2021.

\bibitem{MR3169759}
Jos\'e{} Carmona, Silvia Cingolani, Pedro~J. Mart\'inez-Aparicio, and
  Giuseppina Vannella.
\newblock Regularity and {M}orse index of the solutions to critical quasilinear
  elliptic systems.
\newblock {\em Comm. Partial Differential Equations}, 38(10):1675--1711, 2013.

\bibitem{MR779872}
Giovanna Cerami, Donato Fortunato, and Michael Struwe.
\newblock Bifurcation and multiplicity results for nonlinear elliptic problems
  involving critical {S}obolev exponents.
\newblock {\em Ann. Inst. H. Poincar\'e Anal. Non Lin\'eaire}, 1(5):341--350,
  1984.

\bibitem{MR1698537}
D.~C. de~Morais~Filho and M.~A.~S. Souto.
\newblock Systems of {$p$}-{L}aplacean equations involving homogeneous
  nonlinearities with critical {S}obolev exponent degrees.
\newblock {\em Comm. Partial Differential Equations}, 24(7-8):1537--1553, 1999.

\bibitem{MR2371112}
Marco Degiovanni and Sergio Lancelotti.
\newblock Linking over cones and nontrivial solutions for {$p$}-{L}aplace
  equations with {$p$}-superlinear nonlinearity.
\newblock {\em Ann. Inst. H. Poincar\'e Anal. Non Lin\'eaire}, 24(6):907--919,
  2007.

\bibitem{MR2577821}
Ling Ding and Shi-Wu Xiao.
\newblock Multiple positive solutions for a critical quasilinear elliptic
  system.
\newblock {\em Nonlinear Anal.}, 72(5):2592--2607, 2010.

\bibitem{MaPe5}
Said {El Manouni} and Kanishka Perera.
\newblock A bifurcation and multiplicity result for a critical growth elliptic
  problem.
\newblock {\em Applied Mathematics Letters}, 160:109342, 2025.

\bibitem{MR0478189}
Edward~R. Fadell and Paul~H. Rabinowitz.
\newblock Generalized cohomological index theories for {L}ie group actions with
  an application to bifurcation questions for {H}amiltonian systems.
\newblock {\em Invent. Math.}, 45(2):139--174, 1978.

\bibitem{MR1083144}
J.~Garc{\'{\i}}a~Azorero and I.~Peral~Alonso.
\newblock Multiplicity of solutions for elliptic problems with critical
  exponent or with a nonsymmetric term.
\newblock {\em Trans. Amer. Math. Soc.}, 323(2):877--895, 1991.

\bibitem{MR2532793}
Tsing-San Hsu.
\newblock Multiple positive solutions for a critical quasilinear elliptic
  system with concave-convex nonlinearities.
\newblock {\em Nonlinear Anal.}, 71(7-8):2688--2698, 2009.

\bibitem{Pe23}
K.~Perera.
\newblock Abstract multiplicity theorems and applications to critical growth
  problems.
\newblock J. Analyse Math., to appear,
  \href{http://arxiv.org/abs/2308.07901}{\tt arXiv:2308.07901 [math.AP]}.

\bibitem{MR2640827}
Kanishka Perera, Ravi~P. Agarwal, and Donal O'Regan.
\newblock {\em Morse theoretic aspects of {$p$}-{L}aplacian type operators},
  volume 161 of {\em Mathematical Surveys and Monographs}.
\newblock American Mathematical Society, Providence, RI, 2010.

\bibitem{MR3469053}
Kanishka Perera, Marco Squassina, and Yang Yang.
\newblock Bifurcation and multiplicity results for critical {$p$}-{L}aplacian
  problems.
\newblock {\em Topol. Methods Nonlinear Anal.}, 47(1):187--194, 2016.

\end{thebibliography}
\end{document}